\documentclass[reqno,b5paper]{amsart}
\usepackage{amsmath}
\usepackage{amssymb}
\usepackage{amsthm}
\usepackage{enumerate}
\usepackage[mathscr]{eucal}
\theoremstyle{plain}
\newtheorem{thm}{Theorem}[section]

\theoremstyle{definition}
\newtheorem{defn}{Definition}[section]

\begin{document}

\setcounter {page}{1}
\title{A Fixed Point on Generalised Cone Metric Spaces}

\author[S. K. Pal and M. Maity]{ Sudip Kumar Pal* and Manojit Maity**\ }
\thanks{ The research of the First Author was funded by University Grants Commission, Govt. of India through the D. S. Kothari Post Doctoral Fellowship.}
\date{}
\newcommand{\acr}{\newline\indent}
\maketitle
\address{{*\,} Department of Mathematics, University of Kalyani, Kalyani-741235, India.
                Email: sudipkmpal@yahoo.co.in.\acr
           {**\,} Assistant Teacher in Mathematics, Boral High School, Kolkata-700154, India. Email: mepsilon@gmail.com\\}

\maketitle
\begin{abstract}
The aim of this paper is to prove a fixed point theorem on a
generalised cone metric spaces for maps satisfying general
contractive type conditions.
\end{abstract}
\author{}
\maketitle
{ Key words and phrases :} Generalised cone metric space, contractive mapping, fixed point. \\

\textbf {AMS subject classification (2010) :} 54C60, 54H25.  \\

\section{\textbf{Introduction:}} The study of fixed points of
mappings satisfying certain contractive conditions has been very
active area of research. Recently Long-Guang and Xian \cite{9}
generalised the concept of a metric space, by introducing cone
metric spaces, and obtained some fixed point theorem for mappings
satisfying certain contractive conditions. One can consider a
generalisation of a cone metric space by replacing the triangle
inequality by a more general inequality. As such, every cone
metric is a generalised cone metric space but the converse is not
true. However the interesting point to note that two very
important fixed point theorems, namely Banach's fixed point
theorem and Ciric's fixed point theorem have already establised in
such a space. In this paper we continue in this direction and
prove a fixed point theorem of Boyed and Wang \cite{2}, \cite{5} under fairly
general conditions in a generalised cone metric spaces.

\section{\textbf{Main Results}}
Let $E$ be a real Banach space. A nonempty convex closed subset $P
\subset E$ is called a cone in $E$ if it satisfies: \\
(i)   $P$ is closed, nonempty and $P \neq \{0\}$, \\
(ii)  $a,b \in \mathbb{R},~ a,b \geq 0 ~\mbox{and} ~ x,y \in P$
imply that
$ax + by \in P$, \\
(iii) $x \in P ~\mbox{and}~ -x \in P$ imply that $x = 0$. \\
The space $E$ can be partially ordered by the cone $P \subset E$;
i. e. $x \leq y$ if and only if $y-x \in P$. Also we write $x << y
~\mbox{if} ~ y-x \in~ int~ P$, where $int ~ P$ denotes the
interior
of $P$. \\
A cone $P$ is called normal if there exists a constant $K > 0$
such that $0 \leq x \leq y$ implies $\parallel x \parallel \leq
K\parallel y \parallel$. \\
In the following we always suppose that $E$ is a real Banach
space, $P$ is a cone in $E$ and $\leq$ is partial ordering with
respect to $P$.
\begin{defn}
Let $X$ be a nonempty set and let $E$ be a Banach space with cone
$P$ and $d:X^2\rightarrow E$ be a mapping such that for all
$x,y\in X$ and for any $k~(k\geq 2)$ distinct points
$z_1,z_2,......,z_k$ in $X$ each of them different from x
and y, one has \\
1.$\theta\leq d(x,y)$ \mbox{for all} $x,y\in X$, and
$d(x,y)=\theta$ if
and only if $x=y$.\\
2.$d(x,y)=d(y,x)$ \mbox{for all} $x,y \in X$\\
3.$d(x,y)\leq d(x,z_1) + d(z_1,z_2) + .....+ d(z_k,y)$ \mbox{for
all}
$x,y,z_1,z_2,......,z_k$ in $X$.\\
i.e. $\{d(x,z_1)+d(z_1,z_2)+.....+d(z_k,y)-d(x,y)\} \in P$\\
Then we say $(X,d)$ is a generalised cone metric space (or shortly
g.c.m.s.).
\end{defn}
Throughout this section a g.c.m.s. will be denoted by $(X,d)$ (or
sometimes by $X$ only) and $\mathbb{N}$ denote the set of all naturals.\\
Any cone metric space is a g.c.m.s. but the converse is not true
\cite{1}. We first recall some basic definitions.
\begin{defn}
A sequence $\{x_n\}_{n \in \mathbb{N}} \in X$ is said to be a
g.c.m.s. convergent if for every $\varepsilon$ in $E$ with $\theta
< \varepsilon$, there is an $N \in \mathbb{N}$ such that for all
$n>N$, $\varepsilon - d(x_n,x) \in P$ for some fixed $x \in X$.
\end{defn}
\begin{defn}
A sequence $\{x_n\}_{n \in \mathbb{N}} \in X$, is said to be a
g.c.m.s. Cauchy sequence if for every $\varepsilon \in E$ with
$\theta < \varepsilon$, there is an $N \in \mathbb{N}$ such that
for all $m,n
> N$, $\varepsilon-d(x_n,x_m) \in P$.
\end{defn}

We say that a g.c.m.s is complete if every Cauchy sequence in $X$
is convergent in $X$.\\
\begin{defn}
A mapping $T: X \rightarrow X$ is said to be contractive if for
any two points $x,y\in X$, ~ $d(x,y)-d(Tx,Ty)\in P$.
\end{defn}

\begin{defn}
A function $f: E \rightarrow P$ is said to be upper semicontinuous
at $x_0 \in E$ if there exists a neighbourhood $U$ of $x_0$ such
that $f(x)+\epsilon \in P, ~\mbox{for all}~ x\in U$.
\end{defn}
We now prove the following fixed point theorem for Boyd and Wong's
contractive mappings \cite{2}, \cite{5}.

\begin{thm}
Let $X$ be a complete g.c.m.s. and let $T:X \rightarrow X$
satisfies
\begin{eqnarray}
\psi(d(x,y)) - d(Tx,Ty)\in P,
\end{eqnarray}
 where $\psi:\bar{P} \rightarrow E$
is upper semicontinuous from right on $\bar{P}$ (the closure of
the range d) and satisfies $\psi(t) < t ~ for~all ~ t \in
\bar{P}-\{0\}$. Then $T$ has a unique fixed point $x_0$ and $T^nx
\rightarrow x_0 ~ for~all  ~ x \in X$.
\end{thm}

\begin{proof}
Given $x \in X$, define
\begin{eqnarray}
c_n = d(T^nx,T^{n-1}x).
\end{eqnarray}
since $d(Tx,Ty) \leq \psi(d(x,y)) < d(x,y)$, the sequence
$\{c_n\}_{n \in \mathbb{N}}$ ie decreasing. Suppose $c_n
\rightarrow c \in E$. Then if $c > 0$, we have $\psi(c_n)-c_n+1
\in P$. Then $\displaystyle{\limsup_{t\rightarrow c^{+}}}~ \psi(t)
-c \in P$. i.e.~$\psi(t)-c \in P$ which is a contradiction.
Therefore $c_n \rightarrow 0$.

For each $x\in X$, consider the sequence $\{T^nx\}$. First assume
that it is eventually constant. So there is some $n \in
\mathbb{N}$ such
that $T^mx = T^nx = y$ for each $m > n$.\\

Then $T^{m-n}(T^nx) =T^nx$, so denoting $k = m-n$, we have $T^ky =
y~ \mbox{for~all} ~ k \in \mathbb{N}$. It follows that $d(y,Ty) =
d(T^ky,T^{k+1}y) = c_k$ for all k, and since $c_k \rightarrow 0$,
$d(y,Ty) = 0$, so $y = Ty$. Then y is a fixed point of T.\\

If $\{T^nx\}$ is not eventually constant, then it has a
subsequence with pairwise distinct terms. Without loss of
generality, assume that $\{T^nx\}$ is this subsequence. We shall
show that $\{T^nx\}$ is a g.c.m.s. Cauchy sequence. By
contradiction suppose that there is an $\varepsilon>0$ and
sequences $\{m_k\}, \{n_k\}$ of positive integers with $k~\leq~
n_k ~ < m_k $ such that

\begin{eqnarray}
\varepsilon - d(T^{m_k}x,T^{n_k}x) \notin P ~\mbox{for all} ~ k
\in \mathbb{N}.
\end{eqnarray}

Since this is true for all $k \in \mathbb{N}$, we can conclude
that for all $k \in \mathbb{N}$, there will exist $n_k \geq k $
and an infinite number of $m_k > n_k$ for which
\begin{eqnarray}
\frac{\varepsilon}{3} - d(T^{m_k}x,T^{n_k}x) \in P.
\end{eqnarray}

For otherwise let $m_1 > n(K)$ be the highest positive integer for
which (4) holds. Since $c_k \rightarrow {0}$ as $k \rightarrow
\infty$ we can find $m_2 \in \mathbb{N}$ such that
\begin{eqnarray}
 c_k = \frac{\varepsilon}{3} - d(T^kx,T^{k-1}x) \in P  ~\mbox{for all} ~ k \geq
 m_2.
\end{eqnarray}

Now if $m_0 = max(m_1,m_2)$ then for any $i,j > m_0$

$\frac{\varepsilon}{3} - d(T^ix,T^{i+1}x) \in P$ $\Rightarrow$
$\frac{\varepsilon}{3} - d(T^ix,T^jx) \in P$,  if $j = i+1$
$\Rightarrow$ $\varepsilon - d(T^ix,T^jx) \in P$ $\Rightarrow$
$\frac{\varepsilon}{3} - d(T^ix,T^{i+1}x) \in P ;
\frac{\varepsilon}{3} - d(T^{i+1}x,T^nx) \in P ;
\frac{\varepsilon}{3} - d(T^nx,T^jx) \in P$ $\Rightarrow$
$\varepsilon - d(T^ix,,T^jx) \in P ~if~ j > i+1$, which
contradicts (3).

Now in the view of (4) we can choose $m_k$ as the least positive
integer greater than $n_k + 2$ for which
\begin{eqnarray}
d/3 = \frac{\varepsilon}{3} - d(T^{m_k}x,T^{n_k}x) \in P~
\mbox{for all}~ k \in \mathbb{N}.
\end{eqnarray}

Assume that $k \geq m_2$. Now if

$(i)$ $m \geq n+5$ then clearly, $c_m - d(T^mx,T^{m+1}x) \in P$;
$c_{m-1} - d(T^{m-1}x,T^{m-2}x) \in P$; $\frac{\varepsilon}{3} -
d(T^{m-1},T^nx) \in P$. This implies $(c_m + c_{m-1} +
\frac{\varepsilon}{3} - d(T^mx,T^nx) \in P$ ~i.e.~ $(2c_k +
\frac{\varepsilon}{3} - d(T^mx,T^nx) \in P$.

Choose $d_k = (2c_k + \frac{\varepsilon}{3} - d(T^mx,T^nx) \in P$.

$(ii)$ If $m = n+3$ then by (5)

$\frac{\varepsilon}{3} - d(T^{m-2}x,T^nx) \in P$, so $c_k -
d(T^mx,T^{m-1}x) \in P$; $c_k - d(T^{m-1}x,T^{m-2}x) \in P$ and
$\frac{\varepsilon}{3} - d(T^{m-2}x,T^nx) \in P$

implies $(2c_k + \varepsilon/3) - d(T^mx,T^nx) = d_k \in P$.

$(iii)$ If $m = n+4$ then $c_k - d(T^nx,T^{n+1}x) \in P$; $c_k -
d(T^{n+1}x,T^{n+2}x) \in P$; $c_k - d(T^{n+2}x,T^{n+3}x) \in P$;
$\frac{\varepsilon}{3} - d(T^{n+3}x,T^{n+4}x) \in P$. This implies
$(3c_k + \frac{\varepsilon}{3}) - d(T^mx,T^nx) = d_k \in P$.

Hence $\frac{\varepsilon}{3} - d_k \in P$ as $k \rightarrow
\infty$.

Again, $c_k - d(T^mx,T^{m+1}x) \in P$; $c_k - d(T^{m+1}x,T^{n+1}x)
\in P$; ~ i.e. ~

$\psi(d(T^mx,T^nx)) - d(T^{n+1}x,T^nx) \in P$.

Hence $\psi(d_k) - d(T^{n+1}x,T^nx) \in P$. Which shows that

\begin{eqnarray}
2c_k + \psi(d_k) - d(T^mx,T^nx) = d_k
\in P
\end{eqnarray}

Thus as $k \rightarrow \infty$ from (7), we obtain
$\psi(\frac{\varepsilon}{3}) - \frac{\varepsilon}{3} \in P$,

which contradicts the given condition since $\varepsilon > 0$.

Therefore in this case $\{T^nx\}$ is a g.c.m.s. Cauchy and as $X$
is complete, $\{T^nx\}$ converges to a point $x_0$ in $X$.

We shall show that $Tx_0 = x_0$. We divide the proof into two
parts. First let $T^nx$ be different from both $x_0$ and $Tx_0$
for any $n \in \mathbb{N}$. Then

$d(Tx_0,T^nx) + d(T^nx,T^{n+1}x) + d(T^{n+1}x,Tx_0) - d(x_0,Tx_0)
\in P$

i.e.~ $d(x_0,T^nx) + c_{n+1} + \psi(d(x_0,T^nx_0)) - d(x_0,Tx_0)
\in P$

hence~ $d(x_0,Tx_0) + c_{n+1} + d(x_0,T^nx_0) - d(x_0,Tx_0) \in
P$. Which gives

$\varepsilon - d(x_0,Tx_0) \in P$ for any $\varepsilon
> 0$ and as $n \rightarrow \infty$.

Which implies $Tx_0 = x_0$.

Next assume that $T^kx = x_0$ or $T^kx = Tx_0$ for some $k \in
\mathbb{N}$.

Obviously then $x_0 \neq x$ and one can easily show that
$\{T^nx_0\}$ is a sequence with the following properties.

$(i)$ $\varepsilon - \displaystyle{\lim_{n \rightarrow \infty}}
d(T^nx_0,x_0) \in P$.

$(ii)$ $x_0 - T^nx_0 \notin P$ for any $n \in \mathbb{N}$.

$(iii)$ $T^rx_0 - T^px_0 \notin P$ for any $p,r \in \mathbb{N}, p
\neq r$.

Hence proceeding the above it immediately follows that $x_0$ is a
fixed point of $T$.

That the fixed point of $T$ is unique easily follows from the
definition of $T$.

\end{proof}

\noindent\textbf{Acknowledgement:} The authors are grateful to the referee for his valuable suggestions
which considerably improved the presentation of the paper. The authors
also acknowledge the kind advice of Prof. Indrajit Lahiri for
the preparation of this paper.\\


\end{document}